\documentclass{birkjour}
\usepackage{amsmath}
\usepackage{amssymb}
\usepackage{amsfonts}
\usepackage{float}
\usepackage{graphicx}
\usepackage{cite}
\usepackage{epstopdf}
\usepackage [latin1]{inputenc}
\usepackage{mathrsfs}
\usepackage{lineno}
\usepackage{listings}
\usepackage{color}
\usepackage{xcolor}
\usepackage{colortbl}
\usepackage{booktabs}
\newtheorem{thm}{Theorem}[section]
\newtheorem{cor}{Corollary}[section]
\newtheorem{lem}{Lemma}[section]

\newtheorem{rem}{Remark}[section]

\newtheorem{ex}{Example}[section]
\theoremstyle{definition}

\theoremstyle{remark}
\numberwithin{equation}{section}
\allowdisplaybreaks[4]
\begin{document}
\title[Lower Bound on Translative Covering Density of Tetrahedra]
{Lower Bound on Translative Covering \\Density of Tetrahedra}

\author{Yiming Li}
\address{
Center for Applied Mathematics \\
Tianjin University\\
Tianjin, 300354\\
P. R. China}
\email{xiaozhuang@tju.edu.cn}

\author{Miao Fu}
\address{
Center for Applied Mathematics \\
Tianjin University\\
Tianjin, 300354\\
P. R. China}
\email{miaofu@tju.edu.cn}

\author{Yuqin Zhang}
\address{
School of Mathematics \\
Tianjin University\\
Tianjin, 300072\\
P. R. China}
\email{yuqinzhang@tju.edu.cn}

\begin{abstract}
In this paper, we present the first nontrivial lower bound on the translative covering density of tetrahedra. To this end, we show the lower bound, in any translative covering of tetrahedra, on the density relative to a given cube. The resulting lower bound on the translative covering density of tetrahedra is $1+1.227\times10^{-3}$.
\end{abstract}
\maketitle
\section{Introduction}
More than 2,300 years ago, Aristotle claimed that \emph{congruent regular tetrahedra can fill the whole space with neither gap nor overlap}. In modern terms, he claimed that regular tetrahedra of given size can tile the three-dimensional Euclidean space $\mathbb{E}^3$. In other words, they can form both a packing and a covering in $\mathbb{E}^3$ simultaneously. Unfortunately, this statement is wrong. Aristotle's mistake was discovered by Regiomontanus in the fifteenth century (see \cite{Lagarias}). Then, two natural questions arose immediately: \emph{What is the density of the densest tetrahedron packing and what is the density of the thinnest tetrahedron covering}? In fact, the packing case was emphasized by D. Hilbert \cite{Hilbert} as a part of his 18th problem. Since then, many scholars, including mathematicians, physicists, and chemical engineers have made contributions (mistakes as well) to tetrahedra packings. For the complicated history, we refer to \cite{Lagarias}. For packings and tetrahedron packings, we refer to \cite{G.Fejes}, \cite{Gravel}, \cite{Hoylman} and \cite{Zong.t}.

Covering is a dual concept of packing. Nevertheless, our knowledge about covering is still very limited. Let $K$ denote a convex body in $\mathbb{E}^n$ and let $O$ denote a centrally symmetric one. In particular, let $B_n$ denote the $n$-dimensional unit ball and let $T_n$ denote an $n$-dimensional regular simplex with unit edges. Let $\theta^c(K)$, $\theta^t(K)$ and $\theta^l(K)$ denote the densities of the thinnest congruent covering, the thinnest translative covering and the thinnest lattice covering of $\mathbb{E}^n$ with $K$, respectively. Clearly,
\begin{align}
1\leq\theta^c(K)\leq\theta^t(K)\leq\theta^l(K). \nonumber
\end{align}

Our covering knowledge is comparatively complete in $\mathbb{E}^2$. In 1939, R. Kershner \cite{Kershner} proved that
\begin{align}
\theta^t(B_2)=\theta^l(B_2)=\frac{2\pi}{\sqrt{27}}.\nonumber
\end{align}
In 1946 and 1950, L. Fejes T\'{o}th \cite{L.Fejes, L.Fejes.} showed that
\begin{align}
\theta^t(O)=\theta^l(O)\leq\frac{2\pi}{\sqrt{27}}\nonumber
\end{align}
holds for all centrally symmetric convex domains, where equality is attained precisely for the ellipses. In 1950, I. F\'{a}ry \cite{Fary} proved that
\begin{align}
\theta^l(K)\leq\frac{3}{2}\nonumber
\end{align}
holds for all convex domains, where equality is attained precisely for the triangles. It is trivial that $\theta^c(T_2)=1$. However, the fact
\begin{align}
\theta^t(T_2)=\frac{3}{2}\nonumber
\end{align}
was proved only in 2010 by J. Januszewski \cite{Januszewski}. More surprisingly, some basic covering problems remain open even in the plane (see \cite{Zong}). For example, we do not yet know if $\theta^t(K)=\theta^l(K)$ holds for all convex domains.

In $\mathbb{E}^3$, except for the five types of parallelohedra which can translatively tile $\mathbb{E}^3$, the only known exact result is
\begin{align}
\theta^l(B_3)=\frac{5\sqrt{5}\pi}{24},\nonumber
\end{align}
which was discovered in 1954 by R. P. Bambah \cite{Bambah}. For tetrahedron coverings, several bounds have been achieved. In the lattice case,
\begin{align}
\frac{2^{16}+1}{2^{16}}\leq\theta^l(T_3)\leq\frac{125}{63},\nonumber
\end{align}
where the upper bound was discovered by C. M. Fiduccia, R. W. Forcade and J. S. Zito \cite{Fiduccia}, R. Dougherty and V. Faber \cite{Dougherty} and R. Forcade and J. Lamoreaux \cite{Forcade} in 1990s, and the lower bound was achieved by F. Xue and C. Zong \cite{Xue} in 2018 by studying the volumes of generalized difference bodies. In 2022, M. Fu, F. Xue and C. Zong \cite{Fu} improved the lower bound to
\begin{align}
\theta^l(T_3)\geq\frac{25}{18},\nonumber
\end{align}
and they also obtained that
\begin{align}
\theta^l(T_4)\geq\frac{343}{264}.\nonumber
\end{align}
In the congruent case, J. H. Conway and S. Torquato \cite{Conway} obtained
\begin{align}
\theta^c(T_3)\leq\frac{9}{8}\nonumber
\end{align}
by constructing a particular tetrahedron covering in 2006. It is surprising that, nothing nontrivial is known about $\theta^t(T_3)$ up to now.

In $\mathbb{E}^n$, from the works of R. P. Bambah, H. S. M. Coxeter, H. Davenport, P. Erd\H{o}s, L. Few, G. L. Watson, and in particular C. A. Rogers (see \cite{RogersB}), we know that
\begin{gather}
\theta^l(K)\leq n^{\log_2\log_e n+c}, \notag\\
\theta^t(K)\leq n\log n+n\log \log n+5n, \notag\ \textrm{and} \\
\frac{n}{e\sqrt{e}}\ll \theta^t(B_n)\leq \theta^l(B_n)\leq c\cdot n(\log_e n)^{\frac{1}{2}\log_2 2\pi e}.\notag
\end{gather}
Since the 1960s, progress in covering is very limited (see \cite{Brass}). In 2018, F. Xue and C. Zong \cite{Xue} obtained that
\begin{align}
\theta^l(T_n)\geq 1+\frac{1}{2^{3n+7}}. \nonumber
\end{align}
Recently, M. Fu, F. Xue and C. Zong \cite{Fu} improved the lower bound to
\begin{align}
\theta^l(T_n)>1+\frac{1}{n}-\frac{1}{(n-1)2^{n-1}}.\nonumber
\end{align}
In 2021, O. Ordentlich, O. Regev and B. Weiss \cite{Ordentlich} improved Rogers' upper bound to
\begin{align}
\theta^l(K)\leq cn^2, \nonumber
\end{align}
where $c$ is a suitable positive constant.

Since $\theta^t(K)$ is invariant under non-singular linear transformations on $K$, we will work on the regular tetrahedron $T$ with vertices $(1, 1, -1)$, $(1, -1, 1)$, $(-1, 1, 1)$, $(-1, -1, -1)$, edge length $2\sqrt{2}$ and volume $\frac{8}{3}$ in this paper. We prove the following result:
\begin{thm}\label{density}
If $T+X$ is a translative covering of $\mathbb{E}^3$, then its density is at least $1+\frac{\sqrt{2}}{1152}$. In other words, we have
\begin{align}
\theta^t(T)\geq1+\frac{\sqrt{2}}{1152}>1+1.227\times10^{-3}.\nonumber
\end{align}
\end{thm}
\section{Geometric Lemmas}
Let $K$ be a convex body and define
\begin{align}
D(K)=\{\mathbf{x}-\mathbf{y}: \mathbf{x}, \mathbf{y}\in K\}. \nonumber
\end{align}
Usually, we call $D(K)$ the \emph{difference body} of $K$. Clearly $D(K)$ is a centrally symmetric convex set centered at the origin $\mathbf{o}$.
\begin{ex}\textbf{(Groemer \cite{Groemer}).} \label{D(2T)Groemer}
Let $2T$ denote the tetrahedron with vertices $(2, 2, -2)$, $(2, -2, 2)$, $(-2, 2, 2)$ and $(-2, -2, -2)$. It is well known that
\begin{align}
D(2T)=\{(x,y,z): \max\{|x|, |y|, |z|\}\leq4, |x|+|y|+|z|\leq8\}, \nonumber
\end{align}
i.e., a cuboctahedron with edge length $4\sqrt{2}$ and volume $\frac{1280}{3}$.
\end{ex}

Let $P$ denote the cube (see Fig. \ref{fig1}) defined by
\begin{align}\label{defineP}
\{(x,y,z): \max\{|x|, |y|, |z|\}\leq4\}.
\end{align}
Clearly,
\begin{align}\label{vol(P)}
D(2T)\subset P\ \textrm{and}\ vol(P)=512.
\end{align}

\begin{figure}[H]
\centering
\includegraphics[height=3.5cm]{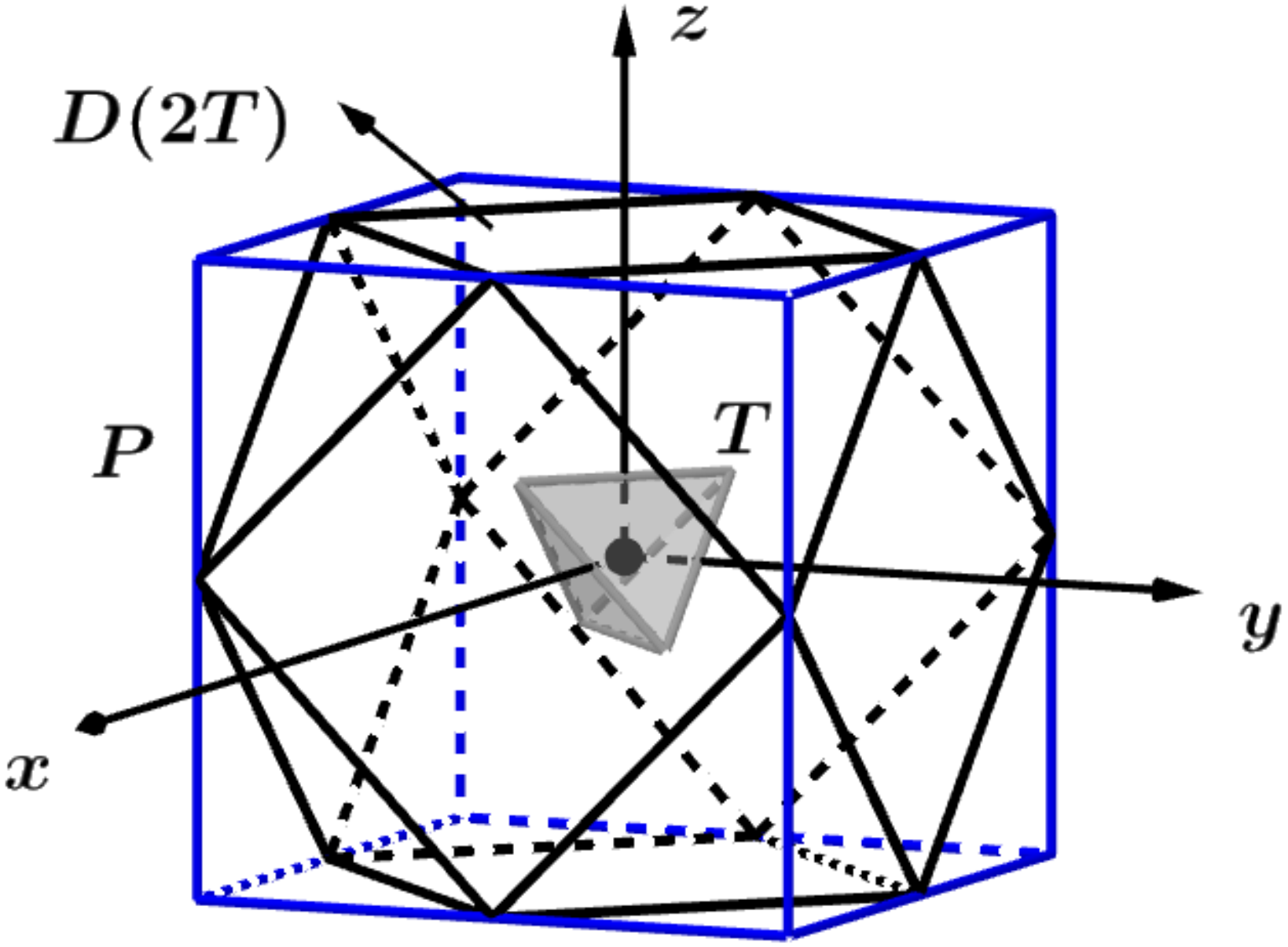}
\caption{The cuboctahedron $D(2T)$ contained in the cube $P$.}
\label{fig1}
\end{figure}

\begin{lem}\label{D(2T)}
If $\mathbf{o}\in T+\mathbf{x}$ and $(T+\mathbf{x})\cap(T+\mathbf{y})\neq\varnothing$, then $T+\mathbf{y}\subset D(2T)$.
\end{lem}
\begin{proof}
Since $\mathbf{o}\in T+\mathbf{x}$, $\mathbf{x}\in -T$. Then
\begin{align}
T+\mathbf{x}\subset T-T=D(T). \nonumber
\end{align}
Without loss of generality, suppose that $(T+\mathbf{x})\cap(T+\mathbf{y})=\mathbf{z}$. Since $\mathbf{z}\in T+\mathbf{x}\subset$ $D(T)$ and $\mathbf{z}\in T+\mathbf{y}$, we have $\mathbf{y}\in -T+\mathbf{z}\subset -T+D(T)$ and therefore
\begin{align}
T+\mathbf{y}\subset T-T+D(T)=D(2T).\nonumber
\end{align}
The lemma is proved.
\end{proof}

\begin{lem}\label{5T}
If $T\cap(T+\mathbf{x})\neq\varnothing$, then $T+\mathbf{x}\subset 5T\cap-7T$.
\end{lem}
\begin{proof}
We claim that $-\frac{1}{3}\mathbf{u}\in T$ if $\mathbf{u}\in T$. Let $\mathbf{v}_1, \mathbf{v}_2, \mathbf{v}_3$ and $\mathbf{v}_4$ denote the vertices of $T$. For any point $\mathbf{u}\in T$, we have
\begin{align}
\mathbf{u}=\alpha_1\mathbf{v}_1+\alpha_2\mathbf{v}_2+\alpha_3\mathbf{v}_3+\alpha_4\mathbf{v}_4, \nonumber
\end{align}
where $\alpha_i\geq0$ for all $i$ and $\alpha_1+\alpha_2+\alpha_3+\alpha_4=1$. Without loss of generality, suppose that $\alpha_4=\max\{\alpha_1, \alpha_2, \alpha_3, \alpha_4\}$. Since $\mathbf{v}_1+\mathbf{v}_2+\mathbf{v}_3+\mathbf{v}_4=\mathbf{o}$, we have
\begin{align}
-\frac{1}{3}\mathbf{u}&=-\frac{1}{3}\left(\alpha_1\mathbf{v}_1+\alpha_2\mathbf{v}_2+\alpha_3\mathbf{v}_3+\alpha_4(-\mathbf{v}_1-\mathbf{v}_2-\mathbf{v}_3)\right) \nonumber\\
&=\frac{1}{3}(\alpha_4-\alpha_1)\mathbf{v}_1+\frac{1}{3}(\alpha_4-\alpha_2)\mathbf{v}_2+\frac{1}{3}(\alpha_4-\alpha_3)\mathbf{v}_3.\nonumber
\end{align}
Since $\frac{1}{3}(\alpha_4-\alpha_i)\geq0$ for all $i$ and the sum of them $\leq1$, combined with the convexity of $T$, we have $-\frac{1}{3}\mathbf{u}\in T$.

Since $T\cap(T+\mathbf{x})\neq\varnothing$, there exist $\mathbf{u}_1, \mathbf{u}_2\in T$ such that $\mathbf{u}_1+\mathbf{x}=\mathbf{u}_2$. Then we have $-\frac{1}{3}\mathbf{u}_1\in T$, $-\frac{1}{3}\mathbf{u}_2\in T$. For any point $\mathbf{u}\in T$, we have $-\frac{1}{3}\mathbf{u}\in T$. By the convexity of $T$, we know that
\begin{gather}
\mathbf{u}+\mathbf{x}=\mathbf{u}+\mathbf{u}_2-\mathbf{u}_1=5\left(\frac{1}{5}\mathbf{u}+\frac{1}{5}\mathbf{u}_2+\frac{3}{5}\left(-\frac{1}{3}\mathbf{u}_1\right)\right)\in 5T, \notag\\
\mathbf{u}+\mathbf{x}=\mathbf{u}+\mathbf{u}_2-\mathbf{u}_1=-7\left(\frac{3}{7}\left(-\frac{1}{3}\mathbf{u}\right)+\frac{3}{7}\left(-\frac{1}{3}\mathbf{u}_2\right)+\frac{1}{7}\mathbf{u}_1
\right)\in -7T.\notag
\end{gather}
Therefore, $T+\mathbf{x}\subset 5T\cap-7T$ if $T\cap(T+\mathbf{x})\neq\varnothing$, the lemma is proved.
\end{proof}

\begin{figure}[H]
\centering
\includegraphics[height=4.5cm]{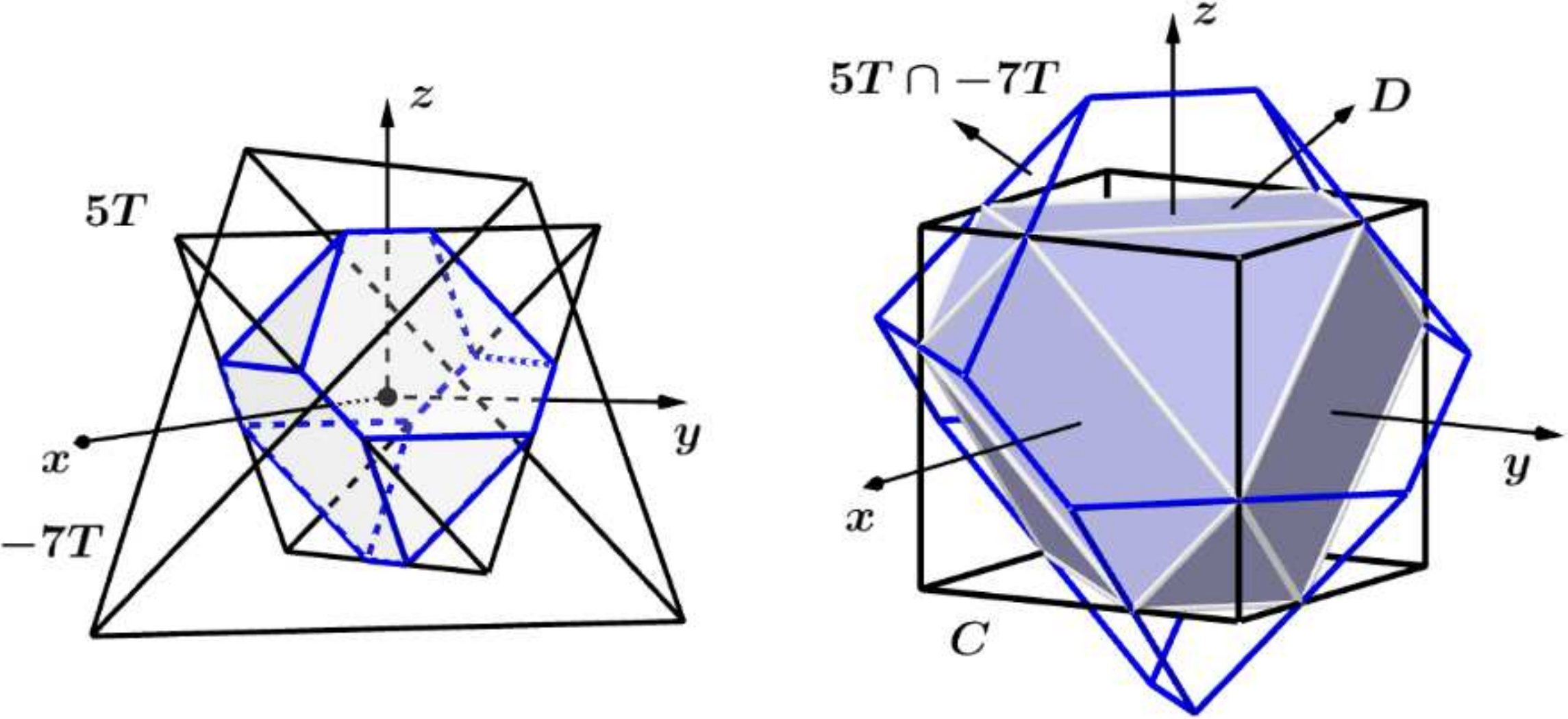}
\caption{$5T\cap-7T$ and $D$.}
\label{fig2}
\end{figure}

Let $C$ denote the cube defined by
\begin{align}
\{(x,y,z): \max\{|x|, |y|, |z|\}\leq3\}, \nonumber
\end{align}
and let
\begin{align}
D=5T\cap-7T\cap C. \nonumber
\end{align}
As shown in Fig. \ref{fig2}, we know that
\begin{align}\label{vol(D)}
vol(D)=vol(C)-4\cdot\frac{4}{3}-4\cdot\frac{32}{3}=168.
\end{align}
The following lemma holds:
\begin{lem}\label{D}
If $T\cap(T+\mathbf{x})\neq\varnothing$, then $T+\mathbf{x}\subset D$.
\end{lem}
\begin{proof}
We know that $T$ can be defined as the set of all points $(x, y, z)$ where
\begin{equation}\label{T=}
\begin{aligned}
x+y+z\leq 1 \\
-x-y+z\leq 1\\
x-y-z\leq 1\\
-x+y-z\leq 1
\end{aligned}
\end{equation}
holds. Since $T\cap(T+\mathbf{x})\neq\varnothing$, there exist $(x_1,y_1,z_1)\in T$ and $(x_2,y_2,z_2)\in T$ such that
\begin{align}
\mathbf{x}=(x_2,y_2,z_2)-(x_1,y_1,z_1). \nonumber
\end{align}
For any point $(x,y,z)\in T$, inserting  $(x,y,z)$ into the two inequalities in (\ref{T=}), $(x_2,y_2,z_2)$ as well, inserting $(x_1,y_1,z_1)$ the remaining two inequalities, and then adding all inequalities together, we have
\begin{align}
|x+(x_2-x_1)|\leq3,\ |y+(y_2-y_1)|\leq3,\ |z+(z_2-z_1)|\leq3.\nonumber
\end{align}
Therefore, $T+\mathbf{x}\subset C$ if $T\cap(T+\mathbf{x})\neq\varnothing$. By Lemma \ref{5T}, $T+\mathbf{x}\subset 5T\cap-7T\cap C$ if $T\cap(T+\mathbf{x})\neq\varnothing$, the lemma is proved.
\end{proof}
\begin{rem}\label{=D}
In fact, by similar arguments it can be deduced that
\begin{align}
\bigcup(T+\mathbf{x})=D,\nonumber
\end{align}
where $T\cap(T+\mathbf{x})\neq\varnothing$. Then by Lemma \ref{D(2T)} and (\ref{vol(P)}),
\begin{align}\label{relation}
\bigcup(T+\mathbf{y})=\bigcup(D+\mathbf{x})\subset D(2T)\subset P,
\end{align}
where $\mathbf{o}\in T+\mathbf{x}$ and $(T+\mathbf{x})\cap(T+\mathbf{y})\neq\varnothing$.
\end{rem}

To prove Theorem \ref{density}, we need the following two results.
\begin{lem}\textbf{(Rogers and Shephard \cite{Rogers}).} \label{homothetic}
An $n$-dimensional convex body $K$ is a simplex if and only if, for any $\mathbf{x}\in int(D(K))$, the intersection $K\cap(K+\mathbf{x})$ is positively homothetic to $K$.
\end{lem}

\begin{cor}\label{triangle}
If $T\cap (T+\mathbf{x})\neq\varnothing$ and $\mathbf{x}\neq \mathbf{o}$, then the following holds:

(1) $T+\mathbf{x}$ intersects at most one vertex of $T$; If $T+\mathbf{x}$ intersects more than one edge (face) of $T$, then it must intersect the vertices (edges) of $T$.

(2) If $int(T+\mathbf{x})$ intersects one vertex (edge or face) of $T$, then $\alpha(T)\cap (T+\mathbf{x})$ is three (two or one) regular triangles with the same edge length.
\end{cor}
\begin{figure}[H]
\centering
\includegraphics[height=2cm]{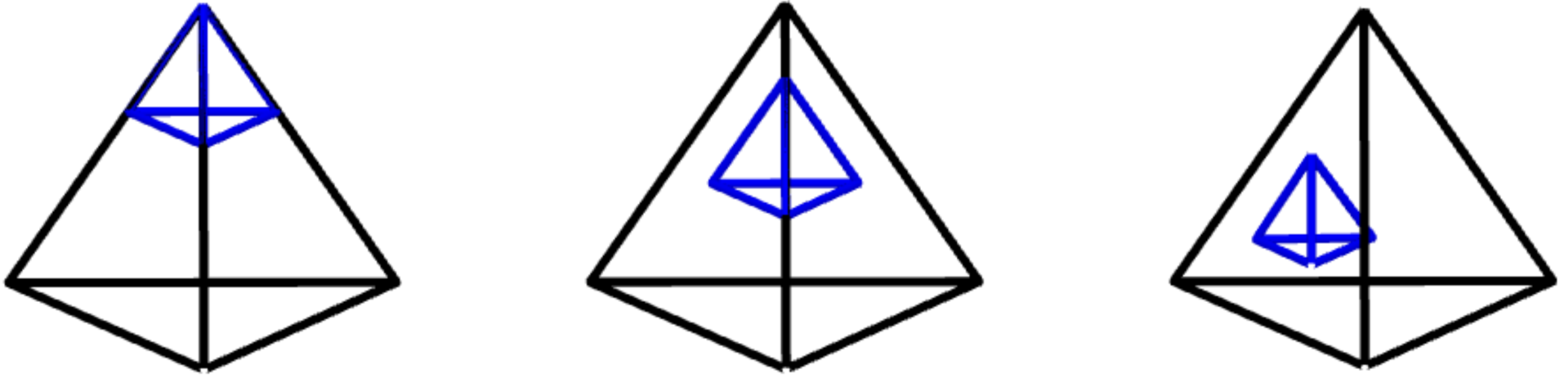}
\caption{Results for Corollary \ref{triangle}(2).}
\label{fig3}
\end{figure}
\section{Translative coverings of tetrahedra}
To study $\theta^t(T)$, the most natural approach is localization. Assume that $X$ is a discrete set of points in $\mathbb{E}^3$ such that $T+X$ is a translative covering of $\mathbb{E}^3$. Let $P$ be the cube defined in (\ref{defineP}). Then define
\begin{align}
\theta(T,X,P)=\frac{\sum\limits_{\mathbf{x}\in X}vol(P\cap(T+\mathbf{x}))}{vol(P)}.\nonumber
\end{align}
Let $\mathfrak{X}$ denote the family of all such sets $X$. We call
\begin{align}
\theta(T,P)=\min_{X\in\mathfrak{X}}\theta(T,X,P)\nonumber
\end{align}
the covering density of $T$ for $P$. Since $P$ is a parallelohedron, clearly,
\begin{align}
\theta^t(T)\geq\theta(T,P).\nonumber
\end{align}

\begin{proof}[\textbf{Proof of Theorem \ref{density}}]
$T+X$ is a translative covering of $\mathbb{E}^3$. Without loss of generality, we suppose that $\mathbf{o}\in T+\mathbf{x}_{m+1}$ and $T+\mathbf{x}_{m+1}$ is intersected by $T+\mathbf{x}_1, T+\mathbf{x}_2, \ldots, T+\mathbf{x}_m$.
By (\ref{relation}), we have
\begin{align}\label{contain}
\bigcup\limits_{i=1}^{m+1}(T+\mathbf{x}_i)\subset D+\mathbf{x}_{m+1}\subset D(2T)\subset P.
\end{align}
We consider two cases.\\
\textbf{Case 1. $m\geq 63$.} By (\ref{vol(P)}), (\ref{vol(D)}) and (\ref{contain}),
\begin{align}
\theta(T,X,P)&=\frac{\sum\limits_{\mathbf{x}\in X}vol\left(P\cap(T+\mathbf{x})\right)}{vol(P)} \nonumber \\
&\geq\frac{vol(P\setminus (D+\mathbf{x}_{m+1}))+\sum\limits_{\mathbf{x}\in X}vol\left((D+\mathbf{x}_{m+1})\cap(T+\mathbf{x})\right)}{vol(P)} \nonumber \\
&\geq1-\frac{21}{64}+\frac{\sum\limits_{i=1}^{m+1}vol\left((D+\mathbf{x}_{m+1})\cap(T+\mathbf{x}_i)\right)}{vol(P)} \nonumber \\
&=\frac{43}{64}+(m+1)\cdot\frac{vol(T)}{vol(P)} \nonumber \\
&\geq \frac{193}{192}. \nonumber
\end{align}\\
\textbf{Case 2. $m\leq 62$.} Let $\partial(T)$ denote the boundary of $T$. Then
\begin{align}\label{boundary}
\partial(T+\mathbf{x}_{m+1})=\bigcup_{i=1}^{m}\left(\partial(T+\mathbf{x}_{m+1})\cap(T+\mathbf{x}_i)\right).
\end{align}
By Lemma \ref{homothetic}, we know that $(T+\mathbf{x}_{m+1})\cap (T+\mathbf{x}_i)$ is a single point or homothetic to $T$. Since $m$ is finite, if $(T+\mathbf{x}_{m+1})\cap (T+\mathbf{x}_1)$ is a single point $\mathbf{u}$, then there must be $i$ such that $\mathbf{u}\in int(T+\mathbf{x}_i)$. So
\begin{align}
\partial(T+\mathbf{x}_{m+1})=\bigcup_{i=2}^{m}\left(\partial(T+\mathbf{x}_{m+1})\cap(T+\mathbf{x}_i)\right),\nonumber
\end{align}
and $m-1\leq 62$ still holds.
Therefore, we suppose that
\begin{align}
(T+\mathbf{x}_{m+1})\cap (T+\mathbf{x}_i)=\lambda_iT+\mathbf{y}_i, 1\leq i\leq m\nonumber
\end{align}
holds for some suitable positive number $\lambda_i$ and a point $\mathbf{y}_i$.

Firstly, let $\mathbf{v}$ be a vertex of $T+\mathbf{x}_{m+1}$. If $\mathbf{v}\in \lambda_1T+\mathbf{y}_1$ and $\mathbf{v}\in \lambda_2T+\mathbf{y}_2$, then we must have $\lambda_1T+\mathbf{y}_1\subset \lambda_2T+\mathbf{y}_2$ or $\lambda_2T+\mathbf{y}_2\subset \lambda_1T+\mathbf{y}_1$, say $\lambda_1T+\mathbf{y}_1\subset \lambda_2T+\mathbf{y}_2$. So
\begin{align}
\partial(T+\mathbf{x}_{m+1})=\bigcup_{i=2}^{m}\left(\partial(T+\mathbf{x}_{m+1})\cap(T+\mathbf{x}_i)\right),\nonumber
\end{align}
and $m-1\leq 62$ still holds.
Therefore, we suppose that each vertex of $T+\mathbf{x}_{m+1}$ is covered exactly once. By Corollary \ref{triangle}(1), these vertices are covered by four different elements in $\{T+\mathbf{x}_1, T+\mathbf{x}_2, \ldots, T+\mathbf{x}_m\}$, say
\begin{align}
T+\mathbf{x}_1, T+\mathbf{x}_2, T+\mathbf{x}_3, T+\mathbf{x}_4.\nonumber
\end{align}
It follows from Corollary \ref{triangle}(2) that the total number of regular triangles obtained by intersecting $T+\mathbf{x}_1, T+\mathbf{x}_2, T+\mathbf{x}_3, T+\mathbf{x}_4$ with $\partial(T+\mathbf{x}_{m+1})$ is 12, and the sum of the areas of these regular triangles is
\begin{align}
3\cdot\frac{\sqrt{3}}{4}\left(\lambda_1^2+\lambda_2^2+\lambda_3^2+\lambda_4^2\right). \nonumber
\end{align}

Secondly, let $e_j$ denote the edge of $T+\mathbf{x}_{m+1}$ and let $E_j$ denote the set consisting of all elements in $\{T+\mathbf{x}_5, T+\mathbf{x}_6, \ldots, T+\mathbf{x}_m\}$ that intersect $e_j$, where $1\leq j\leq6$. Since they do not intersect the vertices of $T+\mathbf{x}_{m+1}$, it follows from Corollary \ref{triangle}(1) that $E_i\cap E_j=\varnothing$ if $i\neq j$. Thus we have $|E_1\cup E_2\cup\cdots\cup E_6|=|E_1|+|E_2|+\cdots+|E_6|$. Without loss of generality, we suppose that
\begin{align}
E_1\cup E_2\cup\cdots\cup E_6=\{T+\mathbf{x}_5, T+\mathbf{x}_6, \ldots, T+\mathbf{x}_{|E_1|+\cdots+|E_6|+4}\}.\nonumber
\end{align}
It follows from Corollary \ref{triangle}(2) that the total number of regular triangles obtained by intersecting $\bigcup E_j$ with $\partial(T+\mathbf{x}_{m+1})$ is
\begin{align}
2\left(|E_1|+|E_2|+\cdots+|E_6|\right), \nonumber
\end{align}
and the sum of the areas of these regular triangles is
\begin{align}
2\cdot\frac{\sqrt{3}}{4}\left(\lambda_5^2+\lambda_6^2+\cdots+\lambda_{|E_1|+\cdots+|E_6|+4}^2\right). \nonumber
\end{align}

Finally, let $f_j$ denote the face of $T+\mathbf{x}_{m+1}$ and let $F_j$ denote the set consisting of all elements in $\{T+\mathbf{x}_{|E_1|+\cdots+|E_6|+5}, T+\mathbf{x}_{|E_1|+\cdots+|E_6|+6}, \ldots,$ $T+\mathbf{x}_m\}$ that intersect $f_j$, where $1\leq j\leq4$. Since they do not intersect the edges of $T+\mathbf{x}_{m+1}$, it follows from Corollary \ref{triangle}(1) that $F_i\cap F_j=\varnothing$ if $i\neq j$. Thus we have
\begin{align}
\left|\bigcup\limits_{i=|E_1|+\cdots+|E_6|+5}^{m} (T+\mathbf{x}_i)\right|=\left|\bigcup\limits_{j=1}^{4}F_j\right|=\sum\limits_{j=1}^{4}\left|F_j\right|, \nonumber
\end{align}
which implies that
\begin{align}\label{m=}
m=4+\sum\limits_{j=1}^{6}|E_j|+\sum\limits_{j=1}^{4}|F_j|.
\end{align}
It follows from Corollary \ref{triangle}(2) that the total number of regular triangles obtained by intersecting $\bigcup F_j$ with $\partial(T+\mathbf{x}_{m+1})$ is
\begin{align}
|F_1|+|F_2|+|F_3|+|F_4|, \nonumber
\end{align}
and the sum of the areas of these regular triangles is
\begin{align}
\frac{\sqrt{3}}{4}\left(\lambda_{|E_1|+\cdots+|E_6|+5}^2+\lambda_{|E_1|+\cdots+|E_6|+6}^2+\cdots+\lambda_m^2\right). \nonumber
\end{align}

Combining the preceding results, we conclude that the total number of regular triangles obtained by intersecting $T+\mathbf{x}_1, T+\mathbf{x}_2, \ldots, T+\mathbf{x}_m$ with $\partial(T+\mathbf{x}_{m+1})$ is
\begin{align}\label{t=}
t=12+2\sum\limits_{j=1}^{6}|E_j|+\sum\limits_{j=1}^{4}|F_j|\leq 2m+4\leq128,
\end{align}
and the sum of the areas of these regular triangles is
\begin{align}
S=\frac{\sqrt{3}}{4}\left(3\sum\limits_{i=1}^{4}{\lambda_i}^2+2\sum\limits_{i=5}^{|E_1|+\cdots+|E_6|+4}{\lambda_i}^2+\sum\limits_{i=|E_1|+\cdots+|E_6|+5}^{m}{\lambda_i}^2\right). \nonumber
\end{align}
From (\ref{boundary}), we know that
\begin{align}\label{area}
S\geq 8\sqrt{3}.
\end{align}
According to Power-Mean Inequality, (\ref{m=}) and (\ref{t=}), we have
\begin{align}
\left(\left(3\sum\limits_{i=1}^{4}{\lambda_i}^3+2\sum\limits_{i=5}^{|E_1|+\cdots+|E_6|+4}{\lambda_i}^3+\sum\limits_{i=|E_1|+\cdots+|E_6|+5}^{m}{\lambda_i}^3\right)/t\right)^{\frac{1}{3}}
\geq\left(S/\frac{\sqrt{3}}{4}t\right)^{\frac{1}{2}} , \nonumber
\end{align}
where $\lambda_i>0$.
Then it follows from (\ref{area}) that
\begin{align}
\sum\limits_{i=1}^{m}{\lambda_i}^3\geq \frac{128\sqrt{2t}}{3t}\geq \frac{16}{3}. \nonumber
\end{align}
Therefore, in this case, we have
\begin{align}
\theta(T,X,P)&=\frac{\sum\limits_{\mathbf{x}\in X}vol\left(P\cap(T+\mathbf{x})\right)}{vol(P)}\nonumber \\
&=1+\frac{\sum\limits_{\mathbf{x}_i,\mathbf{x}_j\in X,i<j}vol\left(P\cap(T+\mathbf{x}_i)\cap(T+\mathbf{x}_j)\right)}{vol(P)} \nonumber \\
&\geq 1+\frac{\sum\limits_{1\leq i<j\leq m+1}vol\left((T+\mathbf{x}_i)\cap(T+\mathbf{x}_j)\right)}{vol(P)}\nonumber \\
&\geq 1+\frac{\sum\limits_{i=1}^{m}vol\left((T+\mathbf{x}_i)\cap(T+\mathbf{x}_{m+1})\right)}{vol(P)}\nonumber \\
&=1+\frac{\frac{\sqrt{2}}{12}(\lambda_1^3+\lambda_2^3+\cdots+\lambda_m^3)}{512}\nonumber \\
&\geq 1+\frac{\sqrt{2}}{1152}. \nonumber
\end{align}

As a conclusion of these two cases, we have
\begin{align}
\theta^t(T)\geq \theta(T,P)=\min_{X\in\mathfrak{X}}\theta(T,X,P)\geq1+\frac{\sqrt{2}}{1152}>1+1.227\times10^{-3}, \nonumber
\end{align}
and the theorem is proved.
\end{proof}

\begin{rem}
By covering the structure with asymptotic $\delta^l(D(2T))=\frac{45}{49}$ (see \cite{Hoylman}) instead of $\frac{vol(D(2T))}{vol(P)}=\frac{5}{6}$, we can slightly improve the lower bound in Theorem \ref{density}. However, since the improvement is not essential, its proof is not include here.
\end{rem}
\section*{Acknowledgement}
\quad \quad This work is supported by the National Natural Science Foundation of China (NSFC11921001, NSFC11801410 and NSFC11971346) and the Natural Key Research and Development Program of China (2018YFA0704701). The authors are grateful to Professor C. Zong for his supervision and discussion.


\begin{thebibliography}{1}

\bibitem{Bambah}R. P. Bambah, On lattice coverings by spheres. \emph{Proc. Nat. Inst. Sci. India} \textbf{20} (1954), 25-52.
\bibitem{Brass}P. Brass, W. Moser and J. Pach, \emph{Research Problems in Discrete Geometry}. Springer, New York, 2005.
\bibitem{Conway}J. H. Conway and S. Torquato, Packing, tiling, and covering with tetrahedra. \emph{Proc. Natl. Acad. Sci. USA} \textbf{103} (2006), 10612-10617.
\bibitem{Dougherty}R. Dougherty and V. Faber, The degree-diameter problem for several varieties of Cayley graphs. I. The abelian case. \emph{SIAM J. Discrete Math}. \textbf{17} (2004), 478-519.
\bibitem{Fary}I. F\'{a}ry, Sur la densit\'{e} des r\'{e}seaux de domaines convexes. \emph{Bull. Soc. Math. France} \textbf{78} (1950), 152-161.
\bibitem{G.Fejes}G. Fejes T\'{o}th and W. Kuperberg, Packing and covering with convex sets. \emph{Handbook of Convex Geometry} (eds. P. M. Gruber and J. M. Wills), 799-860, North-Holland 1993.
\bibitem{L.Fejes}L. Fejes T\'{o}th, Eine Bemerkung \"{u}ber die Bedeckung der Ebene durch Eibereiche mit Mittelpunkt. \emph{Acta Univ. Szeged. Sect. Sci. Math}. \textbf{11} (1946), 93-95.
\bibitem{L.Fejes.}L. Fejes T\'{o}th, Some packing and covering theorems. \emph{Acta Sci. Math. Szeged} \textbf{12} (1950), 62-67.
\bibitem{Fiduccia}C. M. Fiduccia, R. W. Forcade and J. S. Zito, Geometry and diameter bounds of directed Cayley graphs of abelian groups. \emph{SIAM J. Discrete Math}. \textbf{11} (1998), 157-167.
\bibitem{Forcade}R. Forcade and J. Lamoreaux, Lattice-simplex coverings and the 84-shape. \emph{SIAM J. Discrete Math}. \textbf{13} (2000), 194-201.
\bibitem{Fu}M. Fu, F. Xue and C. Zong, Lower Bounds on Lattice Covering Densities of Simplices. arXiv: 2202.06025
\bibitem{Groemer}H. Groemer, \"{U}ber die dichteste gitterf\"{o}rmige Lagerung kongruenter Tetraeder. \emph{Monatsh. Math}. \textbf{66} (1962), 12-15.
\bibitem{Gravel}S. Gravel, V. Elser and Y. Kallus, Upper bound on the packing density of regular tetrahedra and octahedra. \emph{Discrete Comput. Geom}. \textbf{46} (2011), 799-818.
\bibitem{Hilbert}D. Hilbert, Mathematische Probleme. \emph{Arch. Math. Phys}. \textbf{3} (1901), 44-63; \emph{Bull. Amer. Math. Soc}. \textbf{37} (2000), 407-436.
\bibitem{Hoylman}D. J. Hoylman, The densest lattice packing of tetrahedra. \emph{Bull. Amer. Math. Soc.} \textbf{76} (1970), 135-137.
\bibitem{Januszewski}J. Januszewski, Covering the plane with translates of a triangle. \emph{Discrete Comput. Geom.} \textbf{43} (2010), 167-178.
\bibitem{Kershner}R. Kershner, The number of circles covering a set. \emph{Amer. J. Math.} \textbf{61} (1939), 665-671.
\bibitem{Lagarias}J. C. Lagarias and C. Zong, Mysteries in packing regular tetrahedra. \emph{Notices Amer. Math. Soc.} \textbf{59} (2012), 1540-1549.
\bibitem{Ordentlich}O. Ordentlich, O. Regev and B. Weiss, New bounds on the density of lattice coverings. \emph{J. Amer. Math. Soc.} \textbf{35} (2021), 295-308.
\bibitem{Rogers}C. A. Rogers and G. C. Shephard, The difference body of a convex body. \emph{Arch. Math.} \textbf{8} (1957), 220-233.
\bibitem{RogersB}C. A. Rogers, \emph{Packing and Covering}. Cambridge Univ. Press 1964.
\bibitem{Xue}F. Xue and C. Zong, On lattice coverings by simplices. \emph{Adv. Geom.} \textbf{18} (2018), 181-186.
\bibitem{Zong}C. Zong, Packing, covering and tiling in two-dimensional spaces. \emph{Expo. Math.} \textbf{32} (2014), 297-364.
\bibitem{Zong.t}C. Zong, On the translative packing densities of tetrahedra and cuboctahedra. \emph{Adv. Math.} \textbf{260} (2014), 130-190.
\end{thebibliography}
\end{document}